\documentclass[12pt]{amsart}
\usepackage{amssymb}

\usepackage{t1enc}
\usepackage[latin2]{inputenc}
\usepackage{verbatim}
\usepackage{amsmath,amsfonts,amssymb,amsthm}
\usepackage[mathcal]{eucal}
\usepackage{enumerate}

\newtheorem{theorem}{Theorem}[section]
\newtheorem{lemma}[theorem]{Lemma}

\DeclareMathOperator{\supp}{supp}
\numberwithin{equation}{section}

\begin{document}
\author{Nata Gogolashvili, K\'aroly Nagy, George Tephnadze}
\title[Strong convergence theorem]{Strong convergence theorem for  Walsh-Kaczmarz-Fej\'er means}
\thanks{Research supported by projects T\'AMOP-4.2.2.A-11/1/KONV-2012-0051, GINOP-2.2.1-15-2017-00055 and by Shota Rustaveli National Science Foundation grant no. FR-19-676.}

\address{N. Gogolashvili, The University of Georgia, School of Science and Technology, 77a Merab Kostava St, Tbilisi, 0128, Georgia}
\email{nata.gogolashvili@gmail.com}

\address{K. Nagy, Institute of Mathematics and Computer Sciences, University of Ny\'\i regyh\'aza, P.O. Box 166, Ny\'\i regyh\'aza, H-4400 Hungary }
\email{nagy.karoly@nye.hu}

\address{G. Tephnadze, The University of Georgia, School of Science and Technology, 77a Merab Kostava St, Tbilisi, 0128, Georgia}
\email{g.tephnadze@ug.edu.ge }
\date{}

\maketitle

\begin{abstract}
As main result we prove that Fejér means of Walsh-Kaczmarz-Fourier series are uniformly bounded operators from the Hardy martingale space $\ H_{p}$ to the Hardy martingale space $H_{p}$ for $ 0<p\leq 1/2.$
\end{abstract}

\noindent
\textbf{Key words and phrases:} Walsh-Kaczmarz system, Fej\'er means, maximal operator, strong convergence, martingale Hardy space.
\par\noindent
\textbf{2010 Mathematics Subject Classification.} 42C10.

\section{Introduction}
In 1948 $\breve{\text{S}}$neider \cite{Snei} introduced the Walsh-Kaczmarz
system and showed that the inequality
\begin{equation*}
\limsup_{n\to \infty }\frac{D_{n}^{\kappa }(x)}{\log n}\geq C>0
\end{equation*}
holds for the Dirichlet kernels $D_n^\kappa$ almost everywhere. In 1974 Schipp \cite{Sch2} and  Young \cite{Y} proved
that the Walsh-Kaczmarz system is a convergence system. Skvortsov
in 1981 \cite{Sk1} showed that the Fej\'er means with respect to
the Walsh-Kaczmarz system converge uniformly to $f$ for any
continuous functions $f$. G\'at \cite{gat} proved, for any
integrable functions, that the Fej\'er means with respect to the
Walsh-Kaczmarz system converge almost everywhere to the function. He showed that the maximal operator $\sigma^{\kappa,*}$  of Walsh-Kaczmarz-Fej\'er means is of weak type $(1,1)$ and of type $(p,p)$ for all $1<p\leq\infty$.
G\'at's result was generalized by Simon \cite{S2} (see also \cite{si2}), who showed that the maximal operator $\sigma^{\kappa,*}$ is of type $(H_p,L_p)$ for $p>1/2$.

In the endpoint case $p=1/2$  
Goginava \cite{Gog-PM} proved that  the maximal operator is not of type $(H_{1/2},L_{1/2})$. Moreover, in case $p=1/2$
Weisz \cite{We5} showed that the maximal operator is of weak type $(H_{1/2},L_{1/2})$. The investigation of the maximal operator ${\sigma }^{\kappa,*}$ was continued by Goginava and Nagy. In 2011 they proved that the maximal operator $\tilde{\sigma}^{\kappa,*}$ defined by
$$
\tilde{\sigma }^{\kappa,*}:=
\sup_{n\in \mathbf{P}}\frac{|{\sigma}_n^\kappa  f|}{\log^{2}(n+1)}
$$
is bounded from the Hardy space $H_{1/2}$ to the space
$L_{1/2}$ \cite{GN-CZMJ}. They also showed the sharpness of this result. Namely, they showed that for any nondecreasing function $\varphi\colon {\mathbb P}\to [1,\infty)$ satisfying
the condition 
\begin{equation}\label{cond}
\overline{\lim_{n\to\infty}}\frac{\log^2(n+1)}{\varphi (n)}=+\infty
\end{equation}
the modified maximal operator $\tilde{\sigma }^{\kappa,*}_\varphi:=\sup_{n\in \mathbb{P}}
\frac{|\sigma_n^\kappa  f|}{\varphi(n)}$ is not bounded from the Hardy space
$H_{1/2}$ to the space $L_{1/2}$.

The case $0<p<1/2$ was studied by Tephnadze \cite{T1}. Namely, he showed that the maximal operator defined by 
$$
\tilde{\sigma }_p^{\kappa,*}:=
\sup_{n\in \mathbb{P}}\frac{|{\sigma}_n^\kappa  f|}{(n+1)^{1/p-2}}
$$
is bounded from the Hardy space $H_p(G)$ to the space $L_p(G)$ ($0<p<1/2$). He also showed that the sequence $(n+1)^{1/p-2}$ is sharp.

In paper \cite{T2} (see also \cite{PTT,PTTW,T4}) Tephnadze found necessary and sufficient conditions for the
convergence of Walsh-Kaczmarz-Fej\'er means in the terms of  modulus of continuity on
the Hardy spaces $H_p$, when $0 < p <1/2$ and $p=1/2$, separately. We note that the proofs of these two results based on the properties of the modified maximal operators $\tilde{\sigma}^{\kappa,*}$ and $\tilde{\sigma}_p^{\kappa,*}$ (for details see \cite{GN-CZMJ,T1}).

Simon \cite{si3}  proved that there is an
absolute constant $c_{p},$ depending only on $p,$ such that
\begin{equation}\label{1cc}
\overset{\infty }{\underset{k=1}{\sum }}\frac{\left\Vert S^\kappa_{k}f\right\Vert
	_{p}^{p}}{k^{2-p}}\leq c_{p}\left\Vert f\right\Vert _{H_{p}}^{p},
\end{equation}
for all $f\in H_{p}\left( G\right) ,$ where $0<p<1.$

Similar problem for the Walsh-Kaczmarz system in the endpoint case $p=1$ is still open problem, but for the Walsh system analogue of this statement
was proven in the work of Simon \cite{si1}
 (see also \cite{b, b1, gat1, BNPT,tut1}) and for the trigonometric system it is proven by
Smith \cite{sm}.

In the present paper we prove for $0<p\leq 1/2$ and Walsh-Kaczmarz-Fej\'er means, that  there exists a positive constant $c_p$ depending only on $p$, such that 
 $$
 \frac{1}{\log^{[p+1/2]} n}\sum_{m=1}^n 
 \frac{\Vert \sigma_m^\kappa (f)\Vert_{H_p}^p}{m^{2-2p}}\leq 
 c_p \Vert f\Vert_{H_p}^p
 $$
 holds for all $f\in H_p$.
Moreover, we show the sharpness of our main theorem. That is, we state a strong convergence result in the endpoint case $p=1/2$ which was investigated in papers 
\cite{Gog-PM, GN-CZMJ,T2,We5} and we prove a strong convergence result also in the case $0<p<1/2$ with related papers \cite{T1,T2}.
We note that in 2014 analogical Theorems for Walsh-Paley system was reached by Tephnadze \cite{T3},  the two-dimensional case was investigated by Nagy and Tephnadze \cite{NG1, NG2, NG3}.

Now, we give a brief introduction to the theory of dyadic analysis \cite{AVDR,SWSP}.
Let $\mathbb{P}$ denote the set of positive integers, $\mathbb{N:=P}\cup \{ 0\}.$  Let $G$ be the Walsh group. The measure on $G$ is denoted by $\mu$. The elements of $G$ are of the form
$x=\left( x_{0},x_{1},\ldots,x_{k},\ldots\right) $ with coordinates $x_{k}\in
\{0,1\}\left( k\in \mathbb{N}\right) .$ The group operation on $G$
is the coordinate-wise addition modulo 2. A base for the neighborhoods of $G$ can be given in the following
way:
\begin{equation*}
I_{0}\left( x\right) :=G,\quad I_{n}\left( x\right)
:=\left\{ y\in
G:\,y=\left( x_{0},\ldots,x_{n-1},y_{n},y_{n+1},\ldots\right) \right\} ,
\end{equation*}
$\left( x\in G,n\in \mathbb{N}\right) .$
These sets are called dyadic intervals containing $x$. Let $0=\left( 0:i\in \mathbb{N}\right) \in G$ denote the null element of
$G$. Let us set $I_{n}:=I_{n}\left( 0\right) \,\left( n\in
\mathbb{N}\right) $ and 
 $e_{n}:=\left( 0,\ldots,0,1,0,\ldots\right) \in G,$ where the $n$th coordinate  is 1 and the rest are zeros
$\left( n\in \mathbb{N}\right) .$

The $k$th Rademacher function is given by 
\begin{equation*}
r_{k}\left( x\right) :=\left( -1\right) ^{x_{k}}\quad (k\in \mathbb{N},x\in G). 
\end{equation*}

The Walsh-Paley system is defined as the product system of Rademacher functions. Namely,  every natural number $n$ can be  expressed
in the number system of base 2, in the form 
 $n=\sum\limits_{i=0}^{\infty }n_{i}2^{i}$, where $n_{i}\in
\{0,1\}$ is called the $i$th coordinate of $n$ $\left( i\in \mathbb{N}\right) $. 

 Let us define the order $|n|$ of $n$ by 
$\left| n\right| :=\max\{j\in \mathbb{N:}n_{j}\neq 0\}$, that is
$2^{\left| n\right|}\leq n<2^{\left| n\right| +1}.$

The sequence of Walsh-Paley
functions is given by (for details see e.g. \cite{G-E-S, SWSP})
\begin{equation*}
w_{n}\left( x\right) :=\prod\limits_{k=0}^{\infty }\left(
r_{k}\left( x\right) \right) ^{n_{k}}=r_{\left| n\right| }\left(
x\right) \left( -1\right) ^{\sum\limits_{k=0}^{\left| n\right|
		-1}n_{k}x_{k}}=r_{\left| n\right| }\left( x\right)w_{n-2^{\left| n\right|}}\left( x\right)\quad\left( x\in G,n\in \mathbb{P}\right) .
\end{equation*}

The Walsh-Kaczmarz functions are defined by $\kappa_0=1$ and for $n\ge 1$
$$
\kappa_n(x):=r_{\vert n\vert }(x)\prod_{k=0}^{\vert n\vert -1}
(r_{\vert n\vert -1-k}(x))^{n_k}
=r_{\vert n\vert }(x)(-1)^{\sum_{k=0}^{\vert n\vert -1}
	n_kx_{\vert n\vert -1-k}}.
$$

V. A. Skvortsov (see \cite{Sk1}) gave a relation between the
Walsh-Kaczmarz functions and the Walsh-Paley functions by the help of a coordinate-transformation $\tau_A\colon G\to G$ given by
$$
\tau_A(x):=(x_{A-1},x_{A-2},...,x_1,x_0,x_A,x_{A+1},...)$$
for
$A\in
{\mathbb N}.$
By the definition of $\tau_A$, we have
$$
\kappa_n(x)=r_{\vert n\vert }(x)w_{n-2^{|n|}}(\tau_{\vert n\vert}(x)) \quad (n\in {\mathbb N},x\in G).
$$

The Dirichlet kernels and partial sums are defined by
$$
D_{n}^\alpha  := \sum_{k=0}^{n-1} \alpha_k, \quad  S_n^\alpha(f;x):=\sum_{k=0}^{n-1} 
\widehat{f}^\alpha(k)\alpha_k(x)
$$
for both system $\alpha_n =w_n$ $(n\in {\mathbb N})$ and $\alpha_n=\kappa_n$ $(n\in {\mathbb N}),$ separately, Let us set $D_0^\alpha :=0.$ The $2^n$-th Dirichlet kernels have a
closed form (for details see e.g. \cite{G-E-S, SWSP})

\begin{equation}\label{Dir}
D_{2^n}^w(x)=D_{2^n}^\kappa (x)=D_{2^n} (x)=\begin{cases} 0, & \textrm{if } x\not\in I_n\\
2^{n},& \textrm{if } x\in I_n.\end{cases}
\end{equation}

The $n$th Fej\'er mean and kernel of the
Walsh-(Kaczmarz)-Fourier series of a function $f$ is given by
\begin{equation*}
\sigma_{n}^{\alpha}(f;x)=\frac{1}{n}\sum\limits_{j=0}^{n}S_{j}^{\alpha}(f;x),
\quad 
K_{n}^{\alpha}\left( x\right) :=\frac{1}{n}\sum\limits_{k=0}^{n}D_{k}^{\alpha}\left(
x\right) .
\end{equation*}
It is known that (for details see e.g. \cite{G-E-S, SWSP}) there exists a positive constant $C$ such that 
\begin{equation}\label{norm-Fejer}
\Vert K_n^\kappa \Vert_1\leq C \quad\textrm{for all }n\in {\mathbb N}. 
\end{equation}

\section{Hardy spaces and auxiliary propositions}

To prove our main Theorem  we need the following Lemmas and definitions in \cite{GGN-SSMH,Sk1,We1,we3}.
\begin{lemma}[Skvortsov \cite{Sk1}]\label{lemma-Sk}
	For
$n\in {\mathbb P}, x\in G$
	\begin{eqnarray*}
		nK_n^\kappa(x)&=&1+
		\sum_{i=0}^{|n|-1}2^iD_{2^i}(x)
		+\sum_{i=0}^{|n|-1}2^ir_i(x)K_{2^i}^w(\tau_i(x))\\&+& (n-2^{|n|})(D_{2^{|n|}}(x)
		+r_{|n|}(x)K_{n-2^{|n|}}^w(\tau_{|n|}(x))).
	\end{eqnarray*}
\end{lemma}
\begin{lemma}[G\'at \cite{gat}]\label{lemma-Gat}
	Let $A,t\in {\mathbb N}, A>t.$ Suppose that 
	$x\in I_{t}\backslash I_{t+1}.$ Then 
	$$
	K_{2^A}^\omega (x)= \begin{cases} 0, & \textrm{if } x-x_te_t\not\in I_A\\ 
	2^{t-1},& \textrm{if } x-x_te_t\in I_A.\end{cases} 
	$$
	If $x\in I_A$, then $K_{2^A}^\omega (x)=\frac{2^A+1}{2}.$
\end{lemma}

\begin{lemma}[G\'at, Goginava, Nagy \cite{GGN-SSMH}] \label{lemma-GGN-SSMH}
	Let $n<2^{A+1},A>N$ and $x\in I_{N}(
	x_{0},...,x_{m-1},x_{m}=1,$ $0,...,0,x_{l}=1,0,...,0)=:J_N^{m,l}
	,\,l=0,...,N-1,\,m=-1,0,...,l.$ Then
	\begin{equation*}
	\int\limits_{I_{N}}n\left| K_{n}^{w}\left( \tau _{A}\left( x+t\right)
	\right) \right| dt\leq c\frac{2^{A}}{2^{m+l}},
	\end{equation*}
	where
	\begin{eqnarray*}
		I_{N}\left( x_{0},...,x_{m}=1,0,...,0,x_{l}=1,0,...,0\right) :=I_{N}\left( 0,...,0,x_{l}=1,0,...,0\right) 
	\end{eqnarray*} for $m=-1$.
\end{lemma}

The $\sigma $-algebra generated by the dyadic intervals of measure $2^{-k}$ will be denoted by ${\mathcal F}_{k}$ $\left( k\in \mathbb P\right) .$ Denote by $%
f=\left( f^{\left( n\right) },n\in \mathbb P\right) $ a martingale with
respect to $\left( {\mathcal F}_{n},n\in \mathbb P\right) $ (for details see, e. g.
\cite{we3}). The maximal function of a martingale $f$ is defined by
\begin{equation*}
f^{*}=\sup\limits_{n\in \mathbb P}\left| f^{\left( n\right) }\right| .
\end{equation*}

In case $f\in L_{1}\left( G\right) $, the maximal function can also be given
by
\begin{equation*}
f^{*}\left( x\right) =\sup\limits_{n\in \mathbb P}\frac{1}{\mu \left(
	I_{n}(x)\right) }\left| \int\limits_{I_{n}(x)}f\left( u\right) d\mu \left(
u\right) \right| ,\ \ x\in G.
\end{equation*}

For $0<p<\infty $ the Hardy martingale space $H_{p}(G)$ consists of all
martingales for which

\begin{equation*}
\left\| f\right\| _{H_{p}}:=\left\| f^{*}\right\| _{p}<\infty .
\end{equation*}

If $f\in L_{1}\left( G\right) $, then it is easy to show that the sequence $%
\left( S_{2^{n}}f:n\in \mathbb P\right) $ is a martingale. If $f$ is a
martingale, that is $f=(f^{\left( 0\right) },f^{\left( 1\right) },...)$ then
the Walsh-(Kaczmarz)-Fourier coefficients must be defined in a little bit
different way:
\begin{equation*}
\widehat{f}\left( i\right) =\lim\limits_{k\rightarrow \infty
}\int\limits_{G}f^{\left( k\right) }\left( x\right) \alpha _{i}\left(
x\right) d\mu \left( x\right) ,\ \ (\alpha _{i}=w_{i}\text{ or }\kappa _{i}).
\end{equation*}

The Walsh-(Kaczmarz)-Fourier coefficients of $f\in L_{1}\left( G\right) $
are the same as the ones of the martingale $\left( S_{2^{n}}f:n\in \mathbb P\right) $ obtained from $f$.

A useful characterization of the Hardy spaces $H_p$ is the atomic structure. 
A bounded measurable function $a$ is a $p$-atom, if there exists a dyadic interval $I$, such that
\begin{itemize}
	\item[a)] $\int_{I}a d\mu =0$,
	\item[b)] $\Vert a\Vert_\infty \leq \mu(I)^{-1/p}$,
	\item[c)] $\supp a\subset I$.
\end{itemize}

Hardy martingale spaces $H_{p}\left( G\right) $
for $0<p\leq 1$ have atomic characterizations (see e.g. Weisz \cite{We1, we3}):

\begin{lemma}[Weisz \cite{We1}] 	\label{lemma-W2} A martingale $f=\left( f^{\left( n\right) }:n\in \mathbb{N}%
	\right) $ is in $H_{p}\left( 0<p\leq 1\right) $ if and only if there exist a
	sequence $\left( a_{k},k\in \mathbb{N}\right) $ of p-atoms and a sequence $%
	\left( \mu _{k}:k\in \mathbb{N}\right) $ of real numbers such that, for
	every $n\in \mathbb{N},$%
	\begin{equation}
	\qquad \sum_{k=0}^{\infty }\mu _{k}S_{2^{n}}(a_{k})=f^{\left( n\right) },\text{
		\quad a.e.,}  \label{condmart}
	\end{equation}%
	where
	\begin{equation*}
	\qquad \sum_{k=0}^{\infty }\left\vert \mu _{k}\right\vert ^{p}<\infty .
	\end{equation*}%
	Moreover,
	\begin{equation*}
	\left\Vert f\right\Vert _{H_{p}}\backsim \inf \left( \sum_{k=0}^{\infty
	}\left\vert \mu _{k}\right\vert ^{p}\right) ^{1/p},
	\end{equation*}%
	where the infimum is taken over all decomposition of $f=\left( f^{\left(
		n\right) }:n\in \mathbb{N}\right) $ of the form (\ref{condmart}).
\end{lemma}

\begin{lemma}[Weisz \cite{we3}]\label{lemma-weisz}
Suppose that the operator $T$ is $\sigma$-sublinear and for some $0<p\leq 1$
\begin{equation*}
\int\limits_{\overset{-}{I}}\left\vert Ta\right\vert ^{p}d\mu \leq c_{p}<\infty
\end{equation*}%
for every $p$-atom  $a$, where $I$ denotes the support of the atom. 
If $T$ is bounded from $L_\infty$ to $L_\infty $, then
	$$
	\Vert Tf\Vert_p\leq c_p \Vert f\Vert_{H_p} \quad
	\textrm{ for all } f\in H_p.
	$$
\end{lemma}

For a martingale
\begin{equation*}
f=\sum_{n=0}^{\infty }\left( f_{n}-f_{n-1}\right)
\end{equation*}
the conjugate transforms are defined as
\begin{equation*}
\widetilde{f^{\left( t\right) }}=\sum_{n=0}^{\infty }r_{n}\left( t\right)
\left( f_{n}-f_{n-1}\right) ,
\end{equation*}%
where $t\in G$ is fixed. Note that $\widetilde{f^{\left( 0\right) }}=f.$ It
is well-known (see \cite{We1}) that
\begin{equation}\label{conj-trans}
\begin{split}
\left\Vert \widetilde{f^{\left( t\right) }}\right\Vert _{H_{p}\left(
	G\right) } &=\left\Vert f\right\Vert _{H_{p}\left( G\right) },\text{
}\left\Vert f\right\Vert _{H_{p}\left( G\right) }^{p}\sim
\int_{G}\left\Vert \widetilde{f^{\left( t\right) }}\right\Vert _{p}^{p}d\mu(t),
 \\
\text{ }\widetilde{\left( \sigma_{m}^\kappa(f)\right) ^{\left( t\right) }} &=
\sigma_{m}^\kappa(\widetilde{(f)^{\left( t\right) }}).  \end{split}
\end{equation}

\section{Strong convergence theorem and connecting results}

Our main Theorem reads as follows.
\begin{theorem}\label{thm-main1}
 Let $0<p\leq 1/2$. Then there exists a positive constant $c_p$ depending only on $p$, such that 
$$
\frac{1}{\log^{[p+1/2]} n}\sum_{m=1}^n 
\frac{\Vert \sigma_m^\kappa (f)\Vert_{H_p}^p}{m^{2-2p}}\leq 
c_p \Vert f\Vert_{H_p}^p
$$
holds for all $f\in H_p$.
\end{theorem}
\begin{proof}
During the proof of our main theorem we use the notation and some basic result of paper \cite{GN-CZMJ}. 
Let us suppose that
\begin{equation}\label{cond-1}
\frac{1}{\log^{[p+1/2]} n}{\sum_{m=1}^{n}}\frac{\left\Vert \sigma_{m}^\kappa (f)\right\Vert _{p}^{p}}{m^{2-2p}}\leq c_p\left\Vert f\right\Vert
_{H_{p}}^{p} 
\end{equation}
holds for all $f\in H_p$. 
Combining \eqref{conj-trans} and \eqref{cond-1} we have that
\begin{eqnarray}  \label{eq-1}
&&\frac{1}{\log^{[p+1/2]} n}\sum_{m=1}^{n}\frac{\left\Vert \sigma_{m}^\kappa (f)\right\Vert _{H_{p}}^{p}}{m^{2-2p}}\sim \frac{1}{\log^{[p+1/2]} n}%
\sum_{m=1}^{n}\int_{G}\frac{\left\Vert \widetilde{\left( \sigma^\kappa_{m}(f)\right)^{\left( t\right) }}\right\Vert _{p}^{p}}{m^{2-2p}}d\mu(t)
 \\
&=&\int_{G}\frac{1}{\log^{[p+1/2]} n}\sum_{m=1}^{n}\frac{\left\Vert \widetilde{\left(
		\sigma^\kappa_{m}(f)\right) ^{\left( t\right) }}\right\Vert _{p}^{p}}{m^{2-2p}}%
d\mu(t)=\int_{G}\frac{1}{\log^{[p+1/2]} n}\sum_{m=1}^{n}\frac{\left\Vert \sigma^\kappa_{m}(%
	\widetilde{ (f)^{\left( t\right) }})\right\Vert _{p}^{p}}{m^{2-2p}}d\mu(t)  \notag \\
&\leq& c_p\int_{G}\left\Vert \widetilde{ (f)^{\left( t\right) }}
\right\Vert _{H_{p}}^{p}d\mu(t)\sim c_p\int_{G}\left\Vert f\right\Vert
_{H_{p}}^{p}d\mu(t)\sim c_p\left\Vert f\right\Vert _{H_{p}}^{p}.  \notag
\end{eqnarray}

Since $\sigma_{n}^\kappa$ are bounded (see inequality \eqref{norm-Fejer}) from the
space $L_{\infty }$ to the space $L_{\infty }$, by Lemma \ref{lemma-weisz}
it is enough to prove that
\begin{equation*}
\frac{1}{\log^{[p+1/2]} n}\sum_{m=1}^{n}\frac{\left\Vert \sigma^\kappa_{m}(a)\right\Vert
	_{p}^{p}}{m^{2-2p}}<c_p<\infty
\end{equation*}
for every arbitrary $p$-atom $a.$ This leads us to inequality \eqref{cond-1}.

Let $a$ be an arbitrary $p$-atom with support $I$ and $\mu
(I)=2^{-N}$. Without loss of generality, we may assume that $%
I:=I_{N}$. It is easily seen that ${\sigma}_{n}^\kappa(a)=0$
if $n\leq 2^{N}.$ Therefore, we set $n>2^{N}$.

We can write
\begin{equation*}
\begin{split}
\frac{1}{\log^{[p+1/2] }n}&\sum_{m=1}^{n} \frac{\left\Vert \sigma^\kappa_{m}(a)\right\Vert_{p}^{p}}{m^{2-2p}}\leq \frac{1}{\log^{[p+1/2]} n}\sum_{m=2^{N}}^{n}%
\frac{\left\Vert {\sigma^\kappa_{m}(a)}\right\Vert _{p}^{p}}{m^{2-2p}} \\
& \leq \frac{1}{\log^{[p+1/2]} n}\sum_{m=2^{N}}^{n}\int_{I_{N}}\frac{%
	\left\vert {\sigma^\kappa_{m}(a)}\right\vert^{p}}{m^{2-2p}}d\mu 
+\frac{1}{\log^{[p+1/2]} n}\sum_{m=2^{N}}^{n}\int_{\overline{I_{N}}}%
\frac{\left\vert {\sigma^\kappa_{m}(a)}\right\vert ^{p}}{m^{2-2p}}d\mu  \\
& =:I_{1}+I_{2}.
\end{split}%
\end{equation*}%
Inequality \eqref{norm-Fejer} implies
\begin{eqnarray*}
	I_{1} &\leq &\frac{1}{\log^{[p+1/2]} n}\sum_{m=2^{N}}^{\infty }\int_{I_{N}}%
	\frac{\left\vert \sigma^\kappa_{m}(a)\right\vert ^{p}}{m^{2-2p}}d\mu \\
	&\leq &\frac{c_p}{\log^{[p+1/2]} n}\sum_{m=2^{N}}^{\infty }\frac{1}{m^{2-2p}}\left\Vert
	a\right\Vert _{\infty }^{p}/2^{N}\\
	&\leq& \frac{c_p}{\log^{[p+1/2]} n}\sum_{m=2^{N}}^{n}%
	\frac{1}{m^{2-2p}}<c_p
\end{eqnarray*}
for $0<p\leq 1/2$.

Now, we estimate the expression $I_{2}$.
Lemma \ref{lemma-Sk} yields
 that
 \begin{eqnarray*}
 	|{\sigma}_n^\kappa (f)|
 	&=&{|f*K_n^\kappa|}\leq
 	\left| f*\frac{1}{n }\left( 1+ \sum_{i=0}^{|n|-1}2^i D_{2^i}\right)\right|\\
 	&+&
 	\left| f*\frac{1}{n }  \sum_{i=0}^{|n|-1}2^i r_i K_{2^i}^w\circ \tau_{i}\right|
 	+\left| f*\frac{n-2^{|n|}}{n }\left( D_{2^{|n|}}+ r_{|n|} K_{n-2^{|n|}}^w \circ \tau_{|n|}\right)\right|\\
 	&=:&\sum_{i=1}^3 |f* L_n^i|.
 \end{eqnarray*}
It is easily seen that
\begin{equation}\label{eq-2}
\begin{split}
|(a* L_n^i)(x)|&\leq  \int_{I_N}|a(s)||L_n^i(x+s)|d\mu(s)\leq 
\Vert a\Vert_\infty \int_{I_N}|L_n^i(x+s)|d\mu(s)\\
&\leq 2^{N/p}\int_{I_N}|L_n^i(x+s)|d\mu(s)
\end{split}
\end{equation}
for $i=1,2,3$ (and $n>2^N$). For expression $I_2$ we have 
\begin{eqnarray*}
I_2\leq \frac{1}{\log^{[p+1/2]} n}\sum_{m=2^{N}}^{n}
\frac{\sum_{i=1}^3\int_{\overline{I_{N}}}\left\vert (a*L_m^i)(x)\right\vert ^{p}d\mu(x)}{m^{2-2p}}
=:I_2^1+I_2^2+I_2^3.
\end{eqnarray*}
First, we discuss the expression $I_2^1$. 
We decompose  the set $\overline{I_N}$ as 
\begin{equation}\label{decomp}
\overline{I_N}=\bigcup_{j=0}^{N-1} \left( I_j\backslash I_{j+1}\right).
\end{equation}
Set $x\in I_j\backslash I_{j+1}$ and $s\in I_N$, then $x+s\in I_j\backslash I_{j+1}$ for $j=0,...,N-1$. Applying \eqref{eq-2} and \eqref{Dir}, we have
\begin{eqnarray*}
	\int_{I_N} |L_m^1(x+s)|d\mu(s)
	&\leq& \int_{I_N}\frac{1}{m }\left( 1+ \sum_{i=0}^{j}2^i D_{2^i}(x+s)\right)d\mu(s)\\
	&\leq& \frac{c}{m } 2^{2j} 2^{-N}
\end{eqnarray*}
and 
\begin{eqnarray*}
I_2^1&\leq &
\frac{1}{\log^{[p+1/2]} n}\sum_{m=2^{N}}^{n}
\frac{\sum_{t=0}^{N-1}\int_{I_t\backslash I_{t+1}}\left\vert (a*L_m^1)(x)\right\vert ^{p}d\mu(x)}{m^{2-2p}}\\
&\leq&
\frac{1}{\log^{[p+1/2]} n}\sum_{m=2^{N}}^{n}
\frac{\sum_{t=0}^{N-1} \frac{2^N2^{2tp}}{m^p 2^{Np}}2^{-t}}{m^{2-2p}}\\
&\leq&
\frac{c}{\log^{[p+1/2]} n}\sum_{m=2^{N}}^{n}\frac{2^{N(1-p)}}{m^{2-p}}\sum_{t=0}^{N-1}
2^{t(2p-1)}.
\end{eqnarray*}
If $p=1/2$ we have 
$$
I_2^1\leq \frac{c}{\log n}\sum_{m=2^{N}}^{n}\frac{2^{N/2}}{m^{3/2}} N\leq c
$$
and 
if $0<p< 1/2$ 
$$
I_2^1\leq c_p\sum_{m=2^{N}}^{n}\frac{2^{N(1-p)}}{m^{2-p}}\leq c_p.
$$

Second, we discuss the expression $I_2^2$.
We use the disjoint decomposition \eqref{decomp} of $\overline{I_N}$ and we decompose the sets $I_t\backslash I_{t+1}$ as the following disjoint union:
$$
I_t\backslash I_{t+1}=\bigcup_{l=t+1}^N J_t^l,
$$
where $J_t^l:=I_N(0,...,0,x_t=1,0,...,0,x_l=1,x_{l+1},...,x_{N-1})$
for $t<l<N$ and
$J_t^l:=I_N(e_t)$ for $l=N$.
Let $x\in J_t^l$ and $s\in I_N$, then $x+s\in J_t^l$ ($0\leq t<N$, $t<l\leq N$).

For $0\leq t <l<N$, the next inequality showed in \cite[page 681.]{GN-CZMJ}
\begin{equation}\label{ertek-1}
\int_{I_N} |L_m^2(x+s)|d\mu(s)
\leq c\frac{2^{2t}+2^{2l-t}}{m}2^{-N}.
\end{equation}
For $0\leq t<l=N$, it is showed in \cite[page 681.]{GN-CZMJ}
that
\begin{equation}\label{ertek-2}
\int_{I_N} |L_m^2(x+s)|d\mu(s)
\leq c\frac{2^{2t-N}+2^{N-t}+2^{|m|-t}}{m}.
\end{equation}
The decomposition of $\overline{I_N}$ yields
\begin{eqnarray*}
\int_{\overline{I_N}}\left\vert (a*L_m^2)(x)\right\vert ^{p}d\mu(x)
&=&\sum_{t=0}^{N-1}\int_{I_t\backslash I_{t+1}}\left\vert (a*L_m^2)(x)\right\vert ^{p}d\mu(x)\\
&=&\sum_{t=0}^{N-1}\sum_{l=t+1}^{N-1}\int_{J_t^l}\left\vert (a*L_m^2)(x)\right\vert ^{p}d\mu(x)\\
&&+\sum_{t=0}^{N-1}\int_{J_t^N}\left\vert (a*L_m^2)(x)\right\vert ^{p}d\mu(x)
\end{eqnarray*}
and
\begin{eqnarray}
I_2^2&\leq &
\frac{1}{\log^{[p+1/2]} n}\sum_{m=2^{N}}^{n}
\frac{\sum_{t=0}^{N-1}\sum_{l=t+1}^{N-1}\int_{J_t^l}\left\vert (a*L_m^2)(x)\right\vert ^{p}d\mu(x)}{m^{2-2p}}\notag\\
&&+\frac{1}{\log^{[p+1/2]} n}\sum_{m=2^{N}}^{n}
\frac{\sum_{t=0}^{N-1}\int_{J_t^N}\left\vert (a*L_m^2)(x)\right\vert ^{p}d\mu(x)}{m^{2-2p}}\\
&=:&I_2^{2,1}+I_2^{2,2}.\notag
\end{eqnarray}
For $I_2^{2,1}$ we apply inequality \eqref{eq-2} and \eqref{ertek-1} 
\begin{eqnarray*}
I_2^{2,1}&\leq &
\frac{1}{\log^{[p+1/2]} n}\sum_{m=2^{N}}^{n}
\frac{\sum_{t=0}^{N-1}\sum_{l=t+1}^{N-1}
c_p2^N\frac{(2^{2t}+2^{2l-t})^p}{m^p}2^{-Np}	2^{-l}
	}{m^{2-2p}}\notag\\
&\leq&\frac{c_p}{\log^{[p+1/2]} n}\sum_{m=2^{N}}^{n}
2^{N(1-p)}\frac{\sum_{t=0}^{N-1}\sum_{l=t+1}^{N-1}
	(2^{2tp}+2^{(2l-t)p})	2^{-l}}{m^{2-p}}.
\end{eqnarray*}
For $p=1/2$ we get 
$$ 
I_2^{2,1}\leq \frac{c_p}{\log n}2^{N/2}\sum_{m=2^{N}}^{n}
\frac{N}{m^{3/2}}\leq c_p
$$
and for $0<p<1/2$ we have
$$ 
I_2^{2,1}\leq c_p 2^{N(1-p)}
\sum_{m=2^{N}}^{n}\sum_{t=0}^{N-1}\sum_{l=t+1}^{N-1}\frac{
	2^{2tp-l}+2^{l(2p-1)-tp}	}{m^{2-p}}
\leq   2^{N(1-p)}
\sum_{m=2^{N}}^{n} \frac{c_p}{m^{2-p}}\leq c_p.
$$
By inequality \eqref{eq-2} and \eqref{ertek-2} we write 
\begin{eqnarray*}
	I_2^{2,2}&\leq &\frac{c_p}{\log^{[p+1/2]} n}\sum_{m=2^{N}}^{n}
	\frac{\sum_{t=0}^{N-1}
		2^N\frac{(2^{2t-N}+2^{N-t}+2^{|m|-t})^p}{m^p}2^{-N}
		}{m^{2-2p}}\\
&\leq& \frac{c_p}{\log^{[p+1/2]} n}\sum_{m=2^{N}}^{n}
	\frac{\sum_{t=0}^{N-1}
		{(2^{(2t-N)p}+2^{(|m|-t)p})}
	}{m^{2-p}}.
\end{eqnarray*}
We devide the expression $I_2^{2,2}$ into two parts
\begin{eqnarray*}
 \frac{c_p}{\log^{[p+1/2]} n}\sum_{m=2^{N}}^{n}\frac{1}{m^{2-p}}
\sum_{t=0}^{N-1}
 	2^{(2t-N)p} 
 &\leq&  \frac{c_p}{\log^{[p+1/2]} n}\sum_{m=2^{N}}^{n}\frac{1}{m^{2-p}} 2^{Np}\\
 &\leq&  \frac{c_p}{\log^{[p+1/2]} n}2^{N(2p-1)}\leq c_p
\end{eqnarray*}
and
$$
\frac{c_p}{\log^{[p+1/2]} n}\sum_{m=2^{N}}^{n}\frac{1}{m^{2-2p}}
\sum_{t=0}^{N-1}
\frac{2^{(|m|-t)p}}{m^p}
\leq \frac{c_p}{\log^{[p+1/2]} n}\sum_{m=2^{N}}^{n}\frac{1}{m^{2-2p}}
\leq c_p
$$
for $0<p\leq 1/2$.

At last, we discuss the expression $I_2^3$.
We use Lemma \ref{lemma-GGN-SSMH} and the following disjoint decomposition of $\overline{I_N}$:
$$
\overline{I_N}=\bigcup_{l=0}^{N-1}\bigcup_{k=-1}^{l}
J_N^{k,l},
$$
where the set $J_N^{k,l}$ is defined in Lemma \ref{lemma-GGN-SSMH}.

If $x\in \overline{I_N}$ and $s\in I_N$, then $x+s\in \overline{I_N}$ and $D_{2^{|n|}}(x+s)=0$.
Moreover, if $x\in J_N^{k,l}$, then $x+s\in J_N^{k,l}$ and by Lemma  \ref{lemma-GGN-SSMH} we have
\begin{equation}\label{ertek-3}
\int_{I_N} |L_m^3(x+s)|d\mu(s)
\leq c\frac{2^{|m|}}{m2^{l+k}}.
\end{equation}
The decomposition of $\overline{I_N}$ yields
\begin{eqnarray}\label{L3-decomp}
	\int_{\overline{I_N}}\left\vert (a*L_m^3)(x)\right\vert^{p}d\mu(x)
	&=&\sum_{l=0}^{N-1}\sum_{k=-1}^l\int_{J_N^{k,l}}\left\vert (a*L_m^3)(x)\right\vert ^{p}d\mu(x).
\end{eqnarray}
Inequalities \eqref{eq-2}, \eqref{ertek-3} and \eqref{L3-decomp} yield
\begin{eqnarray*}
I_2^3&\leq& \frac{1}{\log^{[p+1/2]} n}\sum_{m=2^{N}}^{n}
\frac{\sum_{l=0}^{N-1}\sum_{k=-1}^l\int_{J_N^{k,l}}\left\vert (a*L_m^3)(x)\right\vert ^{p}d\mu(x)}{m^{2-2p}}\\
&\leq& \frac{c_p}{\log^{[p+1/2]} n}\sum_{m=2^{N}}^{n}
\frac{\sum_{l=0}^{N-1}\sum_{k=-1}^l
2^{-N}2^{(-l-k)p}2^{-(N-k)}
	}{m^{2-2p}}\\
&\leq& \frac{c_p}{\log^{[p+1/2]} n}\sum_{m=2^{N}}^{n}2^{-2N}
\frac{\sum_{l=0}^{N-1}2^{-lp}\sum_{k=-1}^l
	2^{k(1-p)}
}{m^{2-2p}}\\
&\leq& \frac{c_p}{\log^{[p+1/2]} n}\sum_{m=2^{N}}^{n}2^{-2N}
\frac{\sum_{l=0}^{N-1}2^{l(1-2p)}
}{m^{2-2p}}\leq c_p
\end{eqnarray*}
for $0<p\leq 1/2$.
Summarizing our results for expressions $I_1,I_2^1,I_2^{2,1},I_2^{2,2}$ and 
$I_2^3$ we complete the proof of our main Theorem \ref{thm-main1}.
\end{proof}

In the next theorem we show the sharpness of the statement of Theorem \ref{thm-main1} in case $0<p<1/2$.
\begin{theorem}
Let $0<p<1/2$ and $\Phi \colon \mathbb{N}_{+}\rightarrow
\lbrack 1, \infty )$  is any non-decreasing function,
	satisfying the conditions  $\Phi \left( n\right) \uparrow \infty $  and  
\begin{equation*}
\overline{\underset{k\rightarrow \infty }{\lim }}\frac{2^{k\left(
		2-2p\right) }}{\Phi \left( 2^{k}\right) }=\infty .
\end{equation*}
 Then there exists a martingale  $F\in H_{p},$  such that  
\begin{equation*}
\underset{m=1}{\overset{\infty }{\sum }}\frac{\left\Vert \sigma
	_{m}F\right\Vert _{L_{p,\infty }}^{p}}{\Phi \left( m\right) }=\infty .
\end{equation*} 
\end{theorem}

\begin{proof}
Let $\Phi \left( n\right) $
be non-decreasing function, satisfying condition
\begin{equation} \label{12j}
\underset{k\rightarrow \infty }{\lim }\frac{2^{\left( \left\vert n_{k}\right\vert +1\right) \left( 2-2p\right) }}{\Phi \left( 2^{\left\vert n_{k}\right\vert +1}\right) }=\infty.  
\end{equation}

Under condition (\ref{12j}), there exists a sequence $\left\{ \alpha _{k}:%
\text{ }k\geq 0\right\} \subset \left\{ n_{k}:\text{ }k\geq 0\right\} $ such
that
\begin{equation}
\left\vert \alpha _{k}\right\vert \geq 2,\text{ \quad for all \ }k\geq 0
\label{122}
\end{equation}%
and
\begin{equation}
\sum_{\eta =0}^{\infty }\frac{\Phi ^{1/2}\left( 2^{\left\vert \alpha _{\eta
		}\right\vert +1}\right) }{2^{\left\vert \alpha _{\eta }\right\vert \left(
		1-p\right) }}=2^{1-p}\sum_{\eta =0}^{\infty }\frac{\Phi ^{1/2}\left(
	2^{\left\vert \alpha _{\eta }\right\vert +1}\right) }{2^{\left( \left\vert
		\alpha _{\eta }\right\vert +1\right) \left( 1-p\right) }}<c<\infty .
\label{121}
\end{equation}

Let \qquad
\begin{equation*}
F_{n}=\sum_{\left\{ k:\text{ }\left\vert \alpha _{k}\right\vert <n\right\}
}\lambda _{k}a_{k},
\end{equation*}%
where
\begin{equation*}
\lambda _{k}=\frac{\Phi ^{1/2p}\left( 2^{\left\vert \alpha _{k}\right\vert
		+1}\right) }{2^{\left( \left\vert \alpha _{k}\right\vert \right) \left(
		1/p-1\right) }}
\end{equation*}%
and
\begin{equation*}
a_{k}=2^{\left\vert \alpha _{k}\right\vert \left( 1/p-1\right) }\left(
D_{2^{\left\vert \alpha _{k}\right\vert +1}}-D_{2^{\left\vert \alpha
		_{k}\right\vert }}\right) .
\end{equation*}

It is easy to show that the martingale $F=\left( F_{n},\text{ }n\in
\mathbb{N}\right) \in H_{p}.$
Indeed, since
\begin{equation*}
S_{2^{n}}a_{k}=\left\{
\begin{array}{l}
a_{k},\text{ \quad }\left\vert \alpha _{k}\right\vert <n, \\
0,\text{ \quad }\left\vert \alpha _{k}\right\vert \geq n,
\end{array}
\right.
\end{equation*}

\begin{eqnarray*}
\supp (a_{k})=I_{\left\vert \alpha _{k}\right\vert }, \quad 
\int_{I_{\left\vert \alpha _{k}\right\vert }}a_{k}d\mu =0 \quad \text{and} \quad \left\Vert a_{k}\right\Vert _{\infty }\leq 2^{\left\vert \alpha_{k}\right\vert /p}=\left( \mu (\supp a_{k})\right) ^{-1/p}
\end{eqnarray*}
if we apply Lemma \ref{lemma-W2} and (\ref{121}) we conclude that $F\in H_{p}.$

It is easily seen that

\begin{eqnarray} \label{6aa}
\widehat{F}^\kappa(j)  
=
\begin{cases}
\Phi ^{1/2p}\left( 2^{\left\vert \alpha _{k}\right\vert +1}\right) ,&
\text{ if }j\in \left\{ 2^{\left\vert \alpha
	_{k}\right\vert },...,2^{\left\vert \alpha _{k}\right\vert +1}-1\right\} ,%
\text{ }k=0,1,2..., \\
0,&\text{  if }j\notin
\bigcup\limits_{k=0}^{\infty }\left\{ 2^{\left\vert \alpha _{k}\right\vert
},...,2^{\left\vert \alpha _{k}\right\vert +1}-1\right\} .
\end{cases}%
\end{eqnarray}

Let $2^{\left\vert \alpha _{k}\right\vert }<n<2^{\left\vert \alpha
	_{k}\right\vert +1}.$ Using (\ref{6aa}) we can write 
\begin{equation}
\sigma^\kappa_{n}F=\frac{1}{n}\sum_{j=1}^{2^{\left\vert \alpha _{k}\right\vert
}}S^\kappa_{j}F+\frac{1}{n}\sum_{j=2^{\left\vert \alpha _{k}\right\vert
	}+1}^{n}S^\kappa_{j}F=III+IV.  \label{7aa}
\end{equation}

It is simple to show that%
\begin{equation}
S^\kappa_jF=\left\{
\begin{array}{l}
0,\,\ \text{ if \thinspace \thinspace }0\leq j\leq 2^{\left\vert \alpha_{0}\right\vert } \\
\Phi ^{1/2p}\left( 2^{\left\vert \alpha _{0}\right\vert +1}\right) \left(
D^\kappa_j-D_{2^{\left\vert \alpha _{0}\right\vert }}\right) ,\text{
\thinspace\ \thinspace if \thinspace \thinspace \thinspace }2^{\left\vert
\alpha _{0}\right\vert }<j\leq 2^{\left\vert \alpha _{0}\right\vert +1}.
\text{ }
\end{array}
\right.  \label{7aaa}
\end{equation}

Suppose that $2^{\left\vert \alpha _{s}\right\vert }<j\leq 2^{\left\vert
	\alpha _{s}\right\vert +1},$ for some $s=1,2,...,k.$ Then applying (\ref{6aa}%
) we have%
\begin{eqnarray}
S^\kappa_jF &=&\sum_{v=0}^{2^{\left\vert \alpha _{s-1}\right\vert +1}-1}\widehat{F}^\kappa(v)\kappa_{v}+\sum_{v=2^{^{\left\vert \alpha _{s}\right\vert }}}^{j-1}\widehat{F}^\kappa
(v)\kappa_{v}  \label{8aa} \\
&=&\sum_{\eta =0}^{s-1}\sum_{v=2^{\left\vert \alpha _{\eta }\right\vert
}}^{2^{\left\vert \alpha _{\eta }\right\vert +1}-1}\widehat{F}^\kappa
(v)\kappa_{v}+\sum_{v=2^{\left\vert \alpha_{s}\right\vert }}^{j-1}\widehat{F}^\kappa
(v)\kappa_{v}  \notag \\
&=&\sum_{\eta =0}^{s-1}\sum_{v=2^{\left\vert \alpha _{\eta }\right\vert
}}^{2^{\left\vert \alpha _{\eta }\right\vert +1}-1}\Phi ^{1/2p}\left(
2^{\left\vert \alpha _{\eta }\right\vert +1}\right)\kappa_{v}+\Phi ^{1/2p}\left(2^{\left\vert \alpha_{s}\right\vert +1}\right) \sum_{v=2^{\left\vert \alpha_{s}\right\vert }}^{j-1}\kappa_{v}  \notag \\
&=&\sum_{\eta =0}^{s-1}\Phi ^{1/2p}\left( 2^{\left\vert \alpha _{\eta}\right\vert +1}\right)\left( D_{2^{\left\vert \alpha _{\eta }\right\vert+1}}-D_{2^{\left\vert \alpha _{\eta }\right\vert }}\right)  \notag \\
&&+\Phi ^{1/2p}\left( 2^{\left\vert \alpha _{s}\right\vert +1}\right) \left(D^\kappa_j-D_{2^{\left\vert \alpha _{s}\right\vert }}\right) .  \notag
\end{eqnarray}

Let $2^{\left\vert \alpha _{s}\right\vert +1}\leq j\leq 2^{\left\vert \alpha
	_{s+1}\right\vert },$ $s=0,1,...k-1.$ Analogously of (\ref{8aa}) we get%
\begin{equation} \label{10aa}
S^\kappa_{j}F=\sum_{v=0}^{2^{\left\vert \alpha_{s}\right\vert+1}} \widehat{F}^\kappa(v)\kappa_v=\sum_{\eta =0}^{s}\Phi^{1/2p}\left(2^{\left\vert \alpha_{\eta}\right\vert +1}\right)\left(D_{2^{\left\vert\alpha_{\eta} \right\vert+1}}-D_{2^{\left\vert \alpha_{\eta }\right\vert}}\right).  
\end{equation}

Let $x\in I_{2}\left( e_{0}+e_{1}\right) .$ Since (see (\ref{Dir}), Lemma \ref{lemma-Sk} Lemma \ref{lemma-Gat})
\begin{equation}\label{40}
D_{2^n}\left(x\right) =K^\kappa_{2^{n}}\left(x\right)=K^\kappa_{2^{n+1}}\left(x\right)-K^\kappa_{2^{n}}\left(x\right) =0,\text{ for }n\geq 2,
\end{equation}%
from (\ref{122}) and (\ref{8aa}-\ref{10aa}) we obtain%
\begin{equation} \label{9aaa}
III=\frac{1}{n}\sum_{\eta =0}^{k-1}\Phi ^{1/2p}\left( 2^{\left\vert \alpha_{\eta}\right\vert+1}\right)\sum_{v=2^{\left\vert \alpha _{\eta
}\right\vert}+1}^{2^{\left\vert \alpha_{\eta}\right\vert+1}} D^\kappa_{v}\left(x\right)  
\end{equation}%
\begin{equation*}
=\frac{1}{n}\sum_{\eta =0}^{k-1}\Phi ^{1/2p}\left( 2^{\left\vert \alpha
_{\eta }\right\vert +1}\right) \left( 2^{\left\vert \alpha _{\eta
}\right\vert +1}K^\kappa_{2^{\left\vert \alpha _{\eta }\right\vert +1}}\left(x\right)-2^{\left\vert \alpha _{\eta }\right\vert }K^\kappa_{2^{\left\vert \alpha_{\eta }\right\vert }}\left( x\right) \right) =0.
\end{equation*}

Applying (\ref{8aa}), when $s=k$ in $IV$ we have
\begin{eqnarray} \label{9aa}
IV &=&\frac{n-2^{\left\vert \alpha _{k}\right\vert }}{n}\sum_{\eta
=0}^{k-1}\Phi^{1/2p}\left(2^{\left\vert\alpha_{\eta }\right \vert+1}\right) \left( D_{2^{\left\vert \alpha _{\eta }\right\vert
+1}}-D_{2^{\left\vert \alpha _{\eta }\right\vert }}\right)   \\
&&+\frac{\Phi ^{1/2p}\left( 2^{\left\vert \alpha _{k}\right\vert +1}\right)}{n}\sum_{j=2^{_{\left\vert \alpha _{k}\right\vert }}+1}^{n}\left(D^\kappa_j-D_{2^{\left\vert\alpha_k\right\vert }}\right)=IV_1+IV_2.
\notag
\end{eqnarray}

Combining (\ref{122}) and (\ref{40}) we get
\begin{equation}
IV_{1}=0,\text{ \ for \ }x\in I_{2}\left( e_{0}+e_{1}\right) .  \label{8aaaa}
\end{equation}

Let $x\in I_{2}\left(
e_{0}+e_{1}\right) $, $2^{\left\vert \alpha _{k}\right\vert}<n<2^{\left\vert \alpha _{k}\right\vert +1}$ and $n\in \mathbb{A}_{0,2}$ , where  
$\mathbb{A}_{0,2}$ is defined by 
\begin{equation*}
\mathbb{A}_{0,2}:=\left\{ n\in \mathbb{N}:\text{ }n=2^0+2^2+
\sum_{i=3}^{s}n_i2^{i}\right\}.
\end{equation*}
We have that
\begin{equation} \label{30}
D^\kappa_{j+2^{\left\vert \alpha _{k}\right\vert}}=D_{2^{\left\vert \alpha _{k}\right\vert}}+\sum_{j
=2^{\left\vert \alpha _{k}\right\vert}}^{2^{\left\vert \alpha _{k}\right\vert}+j-1}\kappa_{j},\text{ when  }
j<2^{\left\vert \alpha _{k}\right\vert}  
\end{equation}%
from (\ref{40}) and (\ref{30}) we obtain
\begin{eqnarray} \label{31}
\left\vert IV_{2}\right\vert &=&\frac{\Phi ^{1/2p}\left( 2^{\left\vert \alpha_{k}\right\vert +1}\right) }{n}\left\vert \sum_{j=1}^{n-2^{_{\left\vert\alpha _{k}\right\vert }}}\left( D^\kappa_{j+2^{\left\vert \alpha _{k}\right\vert
}}\left( x\right)-D^\kappa_{2^{\left\vert \alpha_{k}\right\vert }}\left( x\right)\right) \right\vert   
\\ \notag
&=&\frac{\Phi ^{1/2p}\left( 2^{\left\vert \alpha _{k}\right\vert +1}\right)}{n}\left\vert \sum_{j=1}^{n-2^{\left\vert _{\alpha _{k}}\right\vert}}
\sum_{l	=2^{\left\vert \alpha _{k}\right\vert}}^{2^{\left\vert \alpha _{k}\right\vert}+j-1}\kappa_{l}
 \right\vert \\ \notag
&=&\frac{\Phi ^{1/2p}\left( 2^{\left\vert \alpha _{k}\right\vert +1}\right)}{n}\left\vert \sum_{j=1}^{n-2^{\left\vert _{\alpha _{k}}\right\vert}}\sum_{i=0}^{j-1}\kappa_{2^{\left\vert \alpha _{k}\right\vert}+i}\left( x\right) \right\vert  
\end{eqnarray}
It is obvious that every $n\in \mathbb{A}_{0,2},$ $2^{\left\vert \alpha _{k}\right\vert
}<n<2^{\left\vert \alpha _{k}\right\vert +1}$ can be expressed as $n=2^{\left\vert \alpha _{k}\right\vert}+4k+1$ and quantity of the members of final double sums are odd numbers. Indeed, quantity of the members can be calculated by the following sum:
$$
\sum_{j=1}^{n-2^{\left\vert _{\alpha _{k}}\right\vert}}j=(n-2^{\left\vert _{\alpha _{k}}\right\vert})(n-2^{\left\vert _{\alpha _{k}}\right\vert}+1)/2=(4k+1)(4k+2)/2=(4k+1)(2k+1).
$$
On the other hand, each members of the sum, which are Kaczmarz function, take values $\pm 1.$ So, it can never be $0$ and the value of such sum can not be less then 1:

$$\left\vert \sum_{j=1}^{n-2^{\left\vert _{\alpha _{k}}\right\vert}}\sum_{i=0}^{j-1}\kappa_i\left( x\right) \right\vert\geq 1, \  \text{for all }  \ x\in G.$$
It follows that
\begin{eqnarray*}
\left\vert IV_{2}\right\vert\geq \frac{\Phi ^{1/2p}\left(2^{\left\vert \alpha _{k}\right\vert+1}\right)}{2^{\left\vert \alpha _{k}\right\vert +1}}.
\end{eqnarray*}

Let $0<p<1/2$ and $n\in \mathbb{A}_{0,2},$ $2^{\left\vert \alpha _{k}\right\vert}<n<2^{\left\vert \alpha _{k}\right\vert +1}$ and $x\in I_{2}\left(
e_{0}+e_{1}\right) $. By combining (\ref{7aa}-\ref{31}) we have

\begin{eqnarray} \label{10aaa}
&&\left\Vert \sigma^\kappa_{n}F\right\Vert _{L_{p,\infty }}^{p}  \\
&\geq &\frac{c_{p}\Phi^{1/2}\left( 2^{\left\vert \alpha _k\right\vert+1}\right)}{2^{p\left(\left\vert \alpha _k\right\vert +1\right)}}\mu
\left\{ x\in I_{2}\left(e_{0}+e_{1}\right) :\text{ }\left\vert \sigma^\kappa_{n}F\right\vert \geq \frac{c_{p}\Phi^{1/2p}\left( 2^{\left\vert \alpha
_{k}\right\vert +1}\right) }{2^{\left\vert \alpha _{k}\right\vert +1}}\right\}   \notag \\
&\geq &\frac{c_{p}\Phi ^{1/2}\left( 2^{\left\vert \alpha _{k}\right\vert+1}\right) }{2^{p\left( \left\vert \alpha _{k}\right\vert +1\right) }}\mu
\left\{ I_{2}\left( e_{0}+e_{1}\right) \right\} \geq \frac{c_{p}\Phi^{1/2}\left( 2^{\left\vert \alpha _{k}\right\vert +1}\right) }{2^{p\left(
\left\vert \alpha _{k}\right\vert +1\right) }}. \notag
\end{eqnarray}

Hence, by using \eqref{12j} and \eqref{10aaa} we get that
\begin{eqnarray*}
&&\underset{n=1}{\overset{\infty }{\sum }}\frac{\left\Vert \sigma^\kappa_{n}F\right\Vert _{L_{p,\infty }}^{p}}{\Phi \left( n\right) }\geq \underset{\left\{ n\in \mathbb{A}_{0,2}:\text{ }2^{\left\vert \alpha _{k}\right\vert}<n<2^{\left\vert \alpha _{k}\right\vert +1}\right\} }{\sum }\frac{%
\left\Vert \sigma^\kappa_{n}F\right\Vert _{L_{p,\infty }}^{p}}{\Phi \left(n\right) } \\
&\geq &\frac{1}{\Phi ^{1/2}\left( 2^{\left\vert \alpha _{k}\right\vert+1}\right) }\underset{\left\{ n\in \mathbb{A}_{0,2}:\text{ }2^{\left\vert
\alpha _{k}\right\vert }<n<2^{\left\vert \alpha _{k}\right\vert +1}\right\} }{\sum }\frac{1}{2^{p\left( \left\vert \alpha _{k}\right\vert +1\right) }} \\
&\geq& \frac{c_{p}2^{\left( 1-p\right) \left( \left\vert \alpha_{k}\right\vert +1\right) }}{\Phi ^{1/2}\left( 2^{\left\vert \alpha_{k}\right\vert +1}\right) }\rightarrow \infty,\text{ \ as \  }
k\rightarrow \infty.
\end{eqnarray*}

The proof is complete.	
	
\end{proof}

\end{document}